\documentclass[11pt]{article}
\usepackage[margin=0.1in,footskip=0.2in]{geometry}
\usepackage[english]{babel}
\usepackage{graphicx}
\usepackage{amssymb}
\usepackage{amsthm}
\usepackage{amssymb}
\usepackage{setspace}
\usepackage{amsmath,amsthm,amscd,amsfonts,mathrsfs,amssymb}
\usepackage{graphicx,enumerate}

\usepackage[all]{xy}
\usepackage{fullpage} %%%!!!
\usepackage{color}
\usepackage{makeidx}
\usepackage{yhmath}
\usepackage{alltt}
\usepackage{setspace}
\usepackage{hyperref}

\usepackage[normalem]{ulem}

\newtheorem{theorem}{Theorem}[section]
\newtheorem{lemma}[theorem]{Lemma}

\newtheorem{proposition}[theorem]{Proposition}
\theoremstyle{definition}
\newtheorem{remark}[theorem]{Remark}
\theoremstyle{definition}

\newcommand{\dd}{{{d}}}
\newcommand{\ssl}[2]{\mathrm{SL}_{#1}(\mathbb{#2})}
\newcommand{\ggl}[2]{\mathrm{GL}_{#1}(\mathbb{#2})}

\newcommand{\C}{\mathbb{C}}

\newcommand{\reff}[1]{(\ref{#1})}
\newcommand{\bra}[1]{{\left(#1\right)}}
\DeclareMathOperator{\sign}{sgn}
\DeclareMathOperator{\adjoint}{Ad}

\definecolor{orange}{rgb}{1,0.5,0}

\newcommand{\nothing}[1]{}

\DeclareMathOperator{\GL}{GL}

\newcommand{\roots}{\Phi}

\newcommand{\rootp}{\Phi^+}

\newcommand{\weylchamber}{\mathcal{P}^{+}}
\DeclareMathOperator{\PGL}{PGL}
\DeclareMathOperator{\SU}{SU}
\DeclareMathOperator{\diag}{diag}
\DeclareMathOperator{\chara}{char}

\newcommand{\sumad}{\underset{n_1/n_2\in \bra{\mathbb{Q}^{\times}}^3 }{\sum_{n_1,n_2}}
\sum_{[n_1,n_2]|n|n_1n_2}}

\usepackage{mathabx}
\newcommand{\paren}[1]{\ensuremath{\left( #1 \right)}}
\newcommand{\abs}[1]{\ensuremath{\left| #1 \right|}}
\newcommand{\set}[1]{\ensuremath{\left\{ #1 \right\}}}
\newcommand{\wbar}[1]{\overline{#1}}

\newcommand{\spec}{\operatorname{spec}}
\renewcommand{\Re}{{\mathop{\mathgroup\symoperators Re}}}
\renewcommand{\Im}{{\mathop{\mathgroup\symoperators Im}}}

\begin{document}

\title{\scshape Plancherel distribution of Satake parameters of Maass cusp forms on $\GL_3$}

\author{Jack Buttcane\thanks{Supported by NSF grant DMS-1601919.}{ }  and Fan Zhou}

\maketitle
\begin{abstract}
We prove an equidistribution result for the Satake parameters of Maass cusp forms on $\GL_3$ with respect to the $p$-adic Plancherel measure by using an application of the Kuznetsov trace formula. The techniques developed in this paper deal with the removal of arithmetic weight $L(1,F,\adjoint)^{-1}$ in the Kuznetsov trace formula on $\GL_3$.
\\
\\
\end{abstract}

%black, blue, brown, cyan, darkgray, gray, green, lightgray, lime, magenta, olive, orange, pink, purple, red, teal, violet, white, yellow.

%\textcolor{orange}{ABCDEFG HIJKLMN OPQRST UVWXYZ }
%\red{red: revision by Jack} 
%\mage{mage: comments by Jack} 
%\blue{blue: revision by Fan }
%\cyan{cyan: comments by Fan}

%$\mathscr{D}$
%\maketitle
%\subfile{sections/introduction}
%%%%%%%%%%%%%%%%%%%%%%%%%%%%%%5
%%%%%%%%%%%%%%%%%%%%%%%%%%%%%%%%%%%%%%%%%%%%%%%%%%%%5
%%%%%%%%%%%%%%%%%%%%%%%%%%%%%%%%%%%%%%%%%%%%%%%%%%%%5
%%%%%%%%%%%%%%%%%%%%%%%%%%%%%%%%%%%%%%%%%%%%%%%%%%%%5
%%%%%%%%%%%%%%%%%%%%%%%%%%%%%%%%%%%%%%%%%%%%%%%%%%%%5
%%%%%%%%%%%%%%%%%%%%%%%%%%%%%%%%%%%%%%%%%%%%%%%%%%%%5
%%%%%%%%%%%%%%%%%%%%%%%%%%%%%%%%%%%%%%%%%%%%%%%%%%%%5
%

\section{Introduction}
Hecke-Maass cusp forms on $\GL_N$ are cuspidal automorphic functions on ${\scriptstyle{\ssl N Z}}\setminus\GL_N(\mathbb{R})/{\scriptstyle{\bra{\operatorname{O}_N(\mathbb{R})\cdot\mathbb{R}^\times} }} $. They correspond to the spherical automorphic representations of $\PGL_N(\mathbb{A_Q})$. 
Their existence is confirmed by 
Selberg in \cite{selberg} for $N=2$,
Miller in \cite{miller} for $N=3$, M\"uller in \cite{mueller} and Lindenstrauss-Venkatesh in \cite{lindenstraussvenkatesh} for all $N\geq 2$.
They all use the Selberg trace formula or its variations and generalizations, such as the Arthur-Selberg trace formula and the pre-trace formula.
A recent typical but highly nontrivial application of the Authur-Selberg trace formula is to evaluate the distribution of the Satake parameter of Maass cusp forms by Matz-Templier in \cite{matztemplier}. It is proved that the Satake parameter at a prime $p$ of a Maass cusp form $F$ is equidistributed with respect to the $p$-adic Plancherel measure on the unitary dual of $\PGL_N(\mathbb{Q}_p)$, as $F$ varies over the cuspidal spectrum, \textit{op. cit}.

Another important tool to study automorphic forms  is the Kuznetsov trace formula (relative trace formula).  
Historically,  on $\GL_2$, the Kuznetsov trace formula is at least as successful as the Selberg trace formula, if not more. 
A recent application of the Kuznetsov trace formula on $\GL_3$ is the breaking of the convexity bound in \cite{blomerbuttcane1}  of $L$-function on $\GL_3$ in the spectral aspect, by Blomer and the first-named author.
Other applications of the Kuznetsov trace formula on $\GL_3$ include \cite{blomer}, \cite{goldfeldkontorovich}, \cite{BBR}, \cite{zhou2}, etc.

In this paper, we show the Plancherel-equidistribution of Satake parameters on $\GL_3$ as an application of the Kuznetsov trace formula, sharpened by \cite{blomerbuttcane1} and \cite{blomerbuttcane2}, combined with weight removal technique of Luo in \cite{luo1} and \cite{luo2}, as well as a formula for adjoint $L$-functions of degree $N^2-1$ for $\GL_N$, which was little known previously.

%The Kuznetsov trace formula and the Selberg trace formula have been both found extremely useful to study automorphic forms on $\GL_2$ and their $L$-functions. Both have great impact over all number theory. However, the later comer Kuznetsov trace formula have been found more user friendly.  Matt-Sound...

A crucial difference between the Kuznetsov trace formula and the Selberg-type formula is the appearance of arithmetic weight $L(1,F,\adjoint)^{-1}$ in the spectral side of the Kuznetsov trace formula.
The Selberg-type formula on $\GL_2$ looks like
$$\sum_f A_f(n)h(\nu_f)=\text{orbital integrals},$$
whereas the Kuznetsov trace formula is
$$\sum_f \frac{A_f(n)\overline{A_f(m)}}{L(1,\adjoint f)}h(\nu_f) =\delta_{m=n}\;\text{main term}+\text{Kloosterman sums},$$
where $\sum\limits_f$ is summing over the cuspidal spectrum of $\GL_2$ and 
$h$ is a test function on the spectral parameters $\nu_f$.
The weight removal technique of Luo in \cite{luo1} and \cite{luo2} is to remove $L(1,\adjoint f)^{-1}$ from the Kuznetsov trace formula.
Luo proves the distribution of $L(1,\adjoint f)$ for $f$ being Maass forms on $\GL_2$ and the Phillips-Sarnak theorem on the Weyl's law, both from the Kuznetsov trace formula on $\GL_2$. 
Such a method is also used in \cite{lauwang} to study the equidistribution of Hecke eigenvalues on $\GL_2$.
The weight removal technique of Luo depends on a precise formula of Shimura for $L(s,\adjoint f)$. Shimura has the famous formula 
for the adjoint $L$-function from $\GL_2$
from \cite{shimura}
$$L(s,\adjoint f)=\zeta(2s)\sum_{n=1}^\infty \frac{A_f(n^2)}{n^s}$$ 
for a Hecke-Maass cusp form (or a modular form) $f$ for $\ssl 2 Z$ with Hecke eigenvalues $A_f(n)$.
The $\GL_N$-analogue of this formula was not clear. Thanks to Joseph Hundley, we have the following formula for $\GL_3$
$$L(s,F,\adjoint)=\zeta(2s)\zeta(3s)\sumad \frac{A_F(n_1,n_2)}{n^s}$$
for a Maass cusp form $F$ for $\ssl 3 Z$ with Fourier coefficients $A_F(*,*)$. 
In the previous formula, $n_1/n_2\in \bra{\mathbb{Q}^\times}^3$ means $n_1/n_2$ is the cube of a rational number.
The formula above can be traced back to the Rankin-Selberg integral of Ginzburg in \cite{ginzburg}, and independently the representation-theoretic works of Kostant, Lusztig and Hesselink. A formula for $\GL_N$ can be found in Theorem \ref{thm:adjoint-lusztig} but it is less explicit.  It will turn out that the adjoint $L$-function has connection with the Kazhdan-Lusztig polynomial. 

Unlike some other applications of the Kuznetsov trace formula, the weight $L(1,F,\adjoint)^{-1}$ is crucial for the equidistribution problems of the Satake parameters. With the weight $L(1,F,\adjoint)^{-1}$, it is proved in \cite{BBR} and \cite{zhou2} that the Satake parameter of $F$ at a fixed prime $p$ is equidistributed with respect to the Sato-Tate measure, as $F$ varies over Maass cusp forms on $\GL_3$. In this paper, we are going to prove that without the weight, the equidistribution is with respect to another measure, the $p$-adic Plancherel measure $\mu_p$ (see \reff{eq:plancherelmeasure}). There is a lot of literature on the equidistribution problems on $\GL_2$, both with and without weight (see \cite[Section 1]{zhou2}). In some sense, our weight removal technique is not to throw away the weight but to absorb it.

The Kuznetsov trace formula on $\GL_3$ used in this paper is a significant improvement over those in \cite{blomer} and \cite{goldfeldkontorovich}. It is comparable to the convexity-breaking paper \cite{blomerbuttcane1}, although the latter is for a different purpose.  The crucial improvements in that paper are the integral formulae \cite[Section 5]{blomerbuttcane1}, following the harmonic analysis of the integral transforms in \cite{buttcane}.  Careful study of these integral formulae yields a sharp cut-off point on the geometric side the Kuznetsov formula.  In this paper, we provide even stronger integral representations \eqref{eq:LongEleIntRepn} and \eqref{eq:w4IntRepn}, easily giving the sharp upper bounds needed for the current work; these integral representations hold largely independent of the particular family of test functions used here.  We expect that such formulae will lead to asymptotic expansions and thereby solutions to previously intractable problems such as removing the maximal Eisenstein contribution (see the discussion after Proposition 3 of \cite{blomerbuttcane2}).

Our method also relies on the functional equation and the analytic properties of $L(s,F,\adjoint)$. For a general automorphic representation $\pi$ of $\GL_3(\mathbb{A}_K)$ and a number field $K$, it is proved by Ginzburg that $L(s,\pi,\adjoint)$ has functional equation and analytic continuation (except finitely many possible but unlikely poles) in \cite{ginzburg}.  For a Maass cusp form $F$ for $\ssl 3 Z$, with the recent work \cite{hundley2} of Hundley, we are able to prove that $L(s,F,\adjoint)$ and its completed $L$-function are holomorphic.

Our work on the Kuznetsov trace formula could have been saved by a good bound toward the Lindel\"of hypothesis of $L(s,F,\adjoint)$ on average over the spectrum (see Lemma \ref{lemma:approximate_dirichlet}). Indeed, the Lindel\"of hypothesis is known on average for a few families of $L$-functions, such as that in \cite{luo1}. However, it is not available here. We compensate that with our improvement in the Kuznetsov trace formula.

\subsection{Equidistribution with and without weight \texorpdfstring{$L(1,F,\adjoint)^{-1}$}{1/L(1,F,Ad)}}
The equidistribution of Hecke eigenvalues of automorphic forms on $\GL_2$ has been studied extensively.
Let $f$ be a cuspidal automorphic form on $\GL_2$. For a majority of $f$, the Sato-Tate conjecture predicts $A_f(p)$ is equidistributed with respect to the Sato-Tate measure, as $p$ varies over all primes. 
Big progress has been made by Harris, Taylor, et al. for modular forms.
From a spectral perspective, it is proved by Sarnak in \cite{sarnak},
Conrey-Duke-Farmer in  \cite{conreydukefarmer}, Serre in \cite{serre} that $A_f(p)$ is equidistributed with respect to the $p$-adic Plancherel measure, as $f$ varies over a family. 
By the Kuznetsov-Bruggeman trace formula, Bruggeman proves that if each $A_f(p)$ is given a weight $L(1,\adjoint f)^{-1}$, then the equidistribution of $A_f(p)$ is changed to the Sato-Tate measure  again (see \cite{bruggeman}).
Later works on the equidistribution problems on  $\GL_2$ are so numerous that we do not include any here.

On $\GL_3$, Bruggeman's analogue is proved in \cite{BBR}, \cite{zhou1}, \cite{zhou2}, by using the Kuznetsov trace formula on $\GL_3$ of \cite{buttcane_ramanujan}, \cite{blomer} and \cite{goldfeldkontorovich}. 

\subsection{Hecke-Maass cusp forms on \texorpdfstring{$\GL_N$}{GLN}}
The dual group of $\PGL_N(\mathbb{Q}_p)$ is $\ssl N C$.
The standard maximal torus of $\ssl N C$ is 
$$\mathsf{T}=\left\{\diag\{\alpha_1, \cdots, \alpha_N\}:\alpha_i \in \C^* \text{ for all }i, \prod_{i=1}^N \alpha_i=1 \right\}\subset \ssl N C$$
The group $\SU_N$ is the standard maximal compact subgroup of $\ssl N C$. The standard maximal torus of $\SU_N$ is 
$$\mathsf{T}_0=\left\{\diag\{\alpha_1, \cdots, \alpha_N\}:\alpha_i \in \C^* \text{ and } |\alpha_i|=1 \text{ for all }i, \prod_{i=1}^N \alpha_i=1 \right\}\subset \ssl N C$$
The Weyl group $W$ ($\cong \mathrm{S}_N$) acts on $\mathsf{T}$ and $\mathsf{T}_0$  by permutating the diagonal entries. 

Let $F$ be a Hecke-Maass cusp form for $\ssl N Z$ with Fourier coefficients $A_F(m_1,\cdots,m_{N-1} )\in \mathbb C$ for $(m_1,\cdots,m_{N-1} )\in \mathbb{N}^{N-1}$, as defined in \cite{goldfeld}. 
Let $\diag\{\alpha_F ^\bra{1}(p),
\cdots, \alpha_F ^\bra{N}(p)
 \}\in \mathsf{T}/W$ 
  be the Satake parameter of $F$ at a prime $p$. 
We have the Shintani formula
$$A_F(p^{m_1},\cdots, p^{m_{N-1}})=S_{m_1,\cdots,m_{N-1}}\bra{\alpha_F ^\bra{1}(p),
\cdots, \alpha_F ^\bra{N}(p)},$$
where $S_{m_1,\cdots,m_{N-1}}(x_1,\cdots,x_N)$ is the Schur polynomial (see \cite[p.233]{goldfeld}).
The $L$-factor of $F$ at a prime $p$ is given by 
$$L_p(s,F):=\prod_{i=1}^N \bra{1-\frac{\alpha_F^\bra{i}(p)}{p^s}}^{-1}$$
and the standard $L$-function of $F$ has the Euler product 
\begin{align*}
L(s,F):&=\sum_{n=1}^\infty \frac{A_F(1,\cdots,1,n)}{n^s}\\
&=\prod_{p\text{ is a prime}}L_p(s,F).
\end{align*}
The generalized Ramanujan-Petersson conjecture predicts
$$\diag\{\alpha_F ^\bra{1}(p),
\cdots, \alpha_F ^\bra{N}(p)
 \}\in \mathsf{T}_0/W,$$
 namely, $\left|\alpha_F ^\bra{i}(p)\right|=1$ for all $i$.

\subsection{Hecke-Maass cusp form for $\ssl 3 Z$ and a family of test functions}
\label{FWOrtho}

Let $(\nu_1,\nu_2)\in \mathbb{C}^2$ and define $\nu_3:=-\nu_1-\nu_2$.
It will be convenient to also use coordinates $\mu = (\mu_1,\mu_2,\mu_3)\in \C^3$, $\mu_1+\mu_2+\mu_3=0$ given by
\begin{equation}\label{eq:mu}
	\mu_1=\nu_1-\nu_3, \qquad \mu_2=\nu_2-\nu_1, \qquad \mu_3=\nu_3-\nu_2.
\end{equation}
Define $\mathfrak{a}^*=\{(\mu_1,\mu_2,\mu_3)\in \mathbb{R}^3|\mu_1+\mu_2+\mu_3=0\}$ 
and 
$\mathfrak{a}^*_\mathbb{C}=\{(\mu_1,\mu_2,\mu_3)\in \mathbb{C}^3|\mu_1+\mu_2+\mu_3=0\}$. The Weyl group $W$ ($\cong \mathrm{S}_3$) acts on $\mathfrak{a}^*$ and $\mathfrak{a}^*_\mathbb{C}$ by permutation of coordinates. Identify bijectively $\nu=(\nu_1,\nu_2)\in \mathbb{C}^2$ with $\mu\in\mathfrak{a}^*_\mathbb{C}$ as in \eqref{eq:mu} and we have $\mathbb{C}^2
\cong\mathfrak{a}^*_\mathbb{C}
$.

Let $F$ be a Hecke-Maass cusp form for $\ssl 3 Z$ of type $\nu_F=(\nu_1(F),\nu_2(F))\in \mathbb{C}^2$ (also $\in \mathfrak{a}^*_\mathbb{C}$). 
Let $(\mu_1(F),\mu_2(F),\mu_3(F))\in \mathfrak{a}^*_\mathbb{C}/W$ be the Langlands parameter of $F$, satisfying \eqref{eq:mu}, 
with the properties $$\mu_1(F)+\mu_2(F)+\mu_3(F)=0$$ and 
\begin{equation}\label{unitary}
\{-\mu_1(F),-\mu_2(F),-\mu_3(F)\}=\{\overline{\mu_1(F)},\overline{\mu_2(F)},\overline{\mu_3(F)}\}.
\end{equation}
 The functional equation for the standard $L$-function $L(s,F)$ is 
$$L(s,F)\prod_{j=1}^3 \Gamma_\mathbb{R}(s-\mu_j(F))=L(1-s,\tilde F)\prod_{j=1}^3 \Gamma_\mathbb{R}(1-s+\mu_j(F)).$$ 

%The convexity bound for $I_F(X)$ is $$I_F(X)\ll \bra{\prod_{i\neq j}|\mu_i(F)-\mu_j(F)|}^{\frac{1	}{4}+\epsilon} X^{-\frac{1}{2}}.$$

%\cyan{$\mathfrak{a}^{\ast}$ need to be defined.}
%{Define $\mathfrak{a}$ to be the Cartan subalgebra of $\mathfrak{pgl}_3(\mathbb{R})$. }

Let $\Omega \subseteq i\mathfrak{a}^{\ast}$ be a compact Weyl-group invariant subset disjoint from the Weyl chamber walls $\{\mu\in \mathfrak{a}^*_\mathbb{C}|w(\mu)=\mu \text{ for some }w\in W \text{ and }w\neq 1\}$ and $T > 1$ a large parameter.
We utilize the test function of \cite[Section 5.2]{blomerbuttcane2}, which approximates the characteristic function on $T\Omega$:
Let $\nu_0 \in \Omega$.
For $\nu\in  \mathfrak{a}^*_{\mathbb C}$, we put $\psi(\nu) = \exp\bra{3\bra{\nu_1^2 +\nu_2^2  + \nu_3^2}}$ and
\[ P(\nu) :=  \prod_{0 \leq n \leq A} \prod_{j=1}^3 \frac{(\nu_j)^2 - \frac{1}{9}(1 + 2n)^2}{T^2} \]
for some large, fixed constant $A$ to compensate poles of the spectral measure in a large tube.  
Now we choose
\begin{equation}\label{eq:def_h_T}
h_T(\nu) :=  P(\nu)^2 \Bigl(\sum_{w \in W}\psi\Bigl(\frac{w(\nu)  -  T\nu_0}{T^{1-\varepsilon}}\Bigr)\Bigr)^2
\end{equation}
for some very small $0 < \varepsilon < 1/2$.
Then $T^\varepsilon$ such functions give a majorant of the characteristic function of $T\Omega$.

 \begin{theorem}\label{thm:main_thm}
 Let $\mu_p$ be the $p$-adic Plancherel measure supported on $\mathsf{T}_0/W$, defined in \eqref{eq:plancherelmeasure}.
Let $h_T$ for $T>1$ be the family of test functions on the spectral parameters of Hecke-Maass cusp forms on $\GL_3$, defined in \eqref{eq:def_h_T}.
For any continuous function $\phi$ on $\mathsf{T}/W$, we have the limit 
$$\lim_{T\to \infty} \frac{ \sum\limits_F\phi\bra{\diag\{\alpha_F ^\bra{1}(p),
\alpha_F ^\bra{2}(p), \alpha_F ^\bra{3}(p)
 \}} \; h_T(\nu_F) }{\sum\limits_F h_T(\nu_F) }= \int_{\mathsf{T}_0/W} \phi\;\dd \mu_p, $$
where $\sum\limits_F$ is a summation over all Hecke-Maass cusp forms for $\ssl 3 Z$ and $\nu_F$ is the spectral parameter of $F$.
\end{theorem}

\section{Representation theory and adjoint $L$-function}

\subsection{Kazhdan-Lusztig polynomial}
Let $G=\ssl N C$ be a semi-simple algebraic group over $\mathbb C$ with Lie algebra $\mathfrak g=\mathfrak{sl}_N(\mathbb{C})$.
%Let $N$ be the nilpotent elements of $\mathfrak g$. 
Let $\roots$ be the root system and $\rootp$ the set of positive roots. 
Let $W$ be its Weyl group.
Let $\rho:=\frac 1 2 \sum_{\alpha\in \rootp}\alpha$ be the half sum of positive roots.
Let $\weylchamber$ be the (positive) Weyl chamber of dominant weights.

 Let $q$ be a symbol. 
For a weight $\beta$ define the Kostant $q$ partition by $$P_q (\beta )= \sum\limits_{\substack{\beta=\sum n(\alpha)\alpha\\\alpha\in \rootp, n(\alpha)\geq 0}} q^{\sum n(\alpha)}.$$
Define Lusztig's $q$ polynomial $$\mathfrak{M}^\beta_\lambda (q)= \sum_{w\in W}(-1)^{\text{length}(w)} P_q (w(\lambda+\rho)-(\beta+\rho)).$$
This is also called the Kostka-Foulkes polynomial or the Kazhdan-Lusztig polynomial (for root systems of $A$ type). Such polynomials were studied extensively but we are only interested in $\mathfrak{M}^0_\lambda(q)$ in this work.

\subsection{Hilbert-Poincar\'e series}
Define $V_\lambda$ as the finite-dimensional complex representation of $G$ for the highest weight $\lambda\in \weylchamber$. By the Weyl character formula, we have 
%$$\chara(V_\lambda)=\frac{\sum_{w\in W}(-1)^{\text{length}(w)} e^{w(\rho+\lambda)}}{\sum_{w\in W}(-1)^{\text{length}(w)} e^{w(\rho)}}.$$
$$\chara(V_\lambda)=\frac{\sum_{w\in W}(-1)^{\text{length}(w)} e^{w(\rho+\lambda)}}{
e^{\rho}\prod_{\beta\in\rootp }\bra{1-e^{-\beta}}}.$$
Recall $\mathfrak{g}=\mathfrak{sl}_N(\mathbb{C})$ and $G$ acts on $\mathfrak{g}$ as the adjoint representation of dimension $N^2-1$.
The symmetric algebra $S(\mathfrak g)=\oplus_{n=0}^\infty \vee^n \mathfrak{g}$ becomes a graded representation of $G$. A fundamental paper of Kostant \cite{kostant} 
proves $S(\mathfrak g)=I\otimes H$ where $G$-invariant part $I$ is a free module (generated by known degrees) and $H=\oplus_{n=0}^\infty H^n$
is the graded module of harmonic polynomials. 
Define $F(V) = \sum_{n=0}^\infty\dim \operatorname{Hom}_\mathfrak{g}(V, H^n)q^n$, as in \cite[\S 2]{kirillov}.
It is proved in  \cite{hesselink} 
$$\sum_{n=0}^\infty \dim \operatorname{Hom}_\mathfrak{g}(V_\lambda, H^n)q^n =  \mathfrak{M}^0_\lambda(q).$$
Hence we have
$$\sum_{n=0}^\infty \chara(H^n) q^n = \sum_{\lambda\in \weylchamber} \mathfrak{M}^0_\lambda(q)\chara (V_\lambda).$$
Combining it with $S(\mathfrak{g})=I\otimes H$, we have for $S(\mathfrak{g})=\oplus_{n=0}^\infty \vee^n \mathfrak{g}$
\begin{equation}\label{eq:poincare_series_adjoint}
\sum_{n=0}^\infty\chara \bra{\vee^n \mathfrak{g} }q^n= 
\bra{\prod_{l=2}^N \bra{1 - q^l}^{-1} }\sum_{\lambda\in \weylchamber} \mathfrak{M}^0_\lambda(q)\chara (V_\lambda).\end{equation}

\subsection{Plancherel measure}
%Let $S_1$ be the unit circle defined as $\{z\in \C: |z|=1\}$.
Let $p$ be a prime number. 
Let $\dd s$ be the normalized Haar measure on $\mathsf{T}_0$.
The Sato-Tate measure is defined on $\mathsf{T}_0/W$ as 
$$\dd \mu_\infty = \frac{1}{|W|}\prod_{\beta\in \roots} \bra{1-e^\beta(s)} \dd s $$
Let $\mu_p$ be the unramified Plancherel measure of $\ssl N C$, which is supported on $\mathsf{T}_0/W$, and it is defined as 
\begin{equation}\label{eq:plancherelmeasure}
\dd \mu_p =  \frac{\mathsf W(p^{-1})}{\prod\limits_{\beta \in \roots} (1-p^{-1}e^\beta (s))}\dd \mu_\infty
\end{equation}
%%%%
with $$\mathsf W(q):=\sum_{w\in W} q^{\textrm{length}(w)}.$$
The formula for $\mu_p$ is due to Macdonald in \cite{macdonald}. 
%\cyan{need page number}

\begin{proposition}[{\cite[(3.4)]{kato}}]
For $\beta \in \weylchamber$ we have 
$$ \int_{\mathsf{T}_0/W} \chara(V_\beta) \dd \mu_p = \mathfrak{M}^0_\beta (p^{-1}) .$$
\end{proposition}

\subsection{Adjoint $L$-function on \texorpdfstring{$\GL_N$}{GLN}}
Let the Weyl chamber $\weylchamber $ be parameterized by $\mathbb{Z}_{\geq 0}^{N-1}$.
Let $\lambda_1$ be the highest weight in $\weylchamber$ for the standard inclusion 
$\ssl N C \hookrightarrow \ggl N C$ (the first minuscule representation).
For $i=2,\cdots, N-1$, let $\lambda_i$ be the highest weight in $\weylchamber$ for the exterior power representation $\wedge^i V_{\lambda_1}$. 
Define a bijective map  $\aleph: \mathbb{Z}_{\geq 0}^{N-1} \to \weylchamber$
by  $$(l_{N-1},\cdots, l_1) \mapsto \sum_{i=1}^{N-1} l_i\lambda_i.$$

Let $F$ be a Hecke-Maass cusp form for $\ssl N Z$. 
Define  $$L(s,F,\adjoint):= \frac{L(s, F\times \tilde F)}{\zeta(s)} $$
as the adjoint $L$-function of $F$. It is a Dirichlet series with Euler product of degree $N^2-1$. The functional equation and holomorphy of the adjoint $L$-function is studied 
in \cite{shimura} for $N=2$, in \cite{ginzburg} for $N=3$, in \cite{bumpginzburg} for $N=4$, and in \cite{ginzburghundley} for $N=5$.
\begin{theorem}\label{thm:adjoint-lusztig}
Let $F$ be a Hecke-Maass cusp form for $\ssl N Z$.
We have the local $L$-function at $p$ 
for $\Re(s)>1$
$$L_p(s,F,\adjoint)=\bra{\prod_{l=2}^N\zeta_p(ls)} \bra{ \sum_{l_1=0}^\infty\cdots
\sum_{l_{N-1}=0}^\infty
 \mathfrak{M}^0_{\aleph(l_{N-1},\cdots,l_1)}(p^{-s}) A_F(p^{l_{N-1}},\cdots, p^{l_1}) }$$
\end{theorem}

\begin{proof}
The theorem follows from \eqref{eq:poincare_series_adjoint}, with the Casselman-Shalika formula for $A_F(p^{l_{N-1}},\cdots, p^{l_1}) $.
\end{proof}

\begin{lemma}\label{lem:speical_case_KL}
For the special case of $N=3$, we have
$$\mathfrak{M}^0_{\aleph(l_2,l_1)}(q)=\begin{cases}
\sum\limits_{i=max\{l_1,l_2\}}^{l_1+l_2} q^i
&,\quad \text{ if } 3|l_1-l_2,
\\
\quad \quad \quad 0&, \quad\text{ otherwise.}
\end{cases}
$$
\end{lemma}
\begin{proof}
Let us assume $l_1\geq l_2$. 
The Weyl group $W$ has six elements. 
Recall the definition $$\mathfrak{M}^0_{\aleph(l_2,l_1)}(q)= \sum_{w\in W}(-1)^{\textrm{length}(w)} P_q (w(\aleph(l_2,l_1)+\rho)-\rho).$$
It is easy to see that $3|l_1-l_2$ is necessary for $\mathfrak{M}^0_{\aleph(l_2,l_1)}(q)$ to be nonzero.
For $w=1$, we have $$P_q(\aleph(l_2,l_1))=\sum\limits_{\frac{2l_1+l_2}{3}\leq i\leq l_1+l_2}q^i.$$
For the Weyl group element $w\in W$ which sends $\lambda_1$ to $\lambda_1$ and $\lambda_2$ to $\lambda_1-\lambda_2$, we have $$(-1)^{\textrm{length}(w)}P_q(w(\aleph(l_2,l_1)+\rho)-\rho)=
-\sum\limits_{\frac{2l_1+l_2}{3}\leq i\leq l_1-1}q^i.$$
For all other $w\in W$, the Kostant $q$ partition $(-1)^{\textrm{length}(w)}P_q(w(\aleph(l_2,l_1)+\rho)-\rho)$ is zero. 
\end{proof}

\subsubsection{The case of $N=3$}
Assume $N=3$ throughout this subsection.  Let $F$ be a Hecke-Maass cusp form for $\ssl 3 Z$.
The functional equation and analytic continuation of $L(s, F, \adjoint) $ is studied by Ginzburg  in \cite{ginzburg} using an integral representation involving Eisenstein series on the exceptional group $G_2$  and by \cite{ginzburgjiang} using a Siegel-Weil identity for $G_2$. 
It is not yet generally known that $L(s,\pi,\adjoint)$ is holomorphic for an automorphic representation $\pi$ of $\GL_3(\mathbb{A}_K)$ and a number field $K$. 
The functional equation for the adjoint $L$-function $L(s,F,\adjoint)$ is
\begin{equation}\label{eq:functional_eq_adjoint}
\Lambda(s,F,\adjoint)=\Lambda(1-s,F,\adjoint), 
\end{equation}
where we have 
$$\Lambda(s,F,\adjoint) = L(s,F,\adjoint)\Gamma_\mathbb{R}^2(s)\prod_{i\neq j} \Gamma_\mathbb{R}(s-\mu_j(F)+\mu_i(F)).
$$
By the recent work of Hundley \cite{hundley2}, we are able to prove that $\Lambda(s,F,\adjoint)$ is holomorphic on the whole complex plane.

%\begin{equation}\label{eq:functional_eq_adjoint}\Gamma(s,F,\adjoint):=L(s,F,\adjoint)\Gamma_\mathbb{R}^2(s)\prod_{i\neq j} \Gamma_\mathbb{R}(s-\mu_j(F)+\mu_i(F))=\Gamma(1-s,F,\adjoint).  \end{equation}

%\cyan{I have some concern here: According to \cite{ginzburgjiang}, $L(s,F,\adjoint)$ is holomorphic for $\Re(s)>1/2$ except $s=1$. I am puzzled why $s=1$ is the exception here. $L(s,F,\adjoint)=\frac{L(s, F\times \tilde{F})}{\zeta(s)} $ and the Rankin-Selberg $ L(s, F\times \tilde{F})$ has a simple pole at $s=1$ from the ``classic'' Rankin-Selberg theory, which is canceled with the pole of $\zeta(s)_{s=1}$. Why is this a difficulty? Is it because of number fields? Or is it because of ramifications for general automorphic forms?}

\begin{theorem}\label{thm:holomorph}
The adjoint $L$-function $L(s,F,\adjoint)$ is holomorphic on the complex plane. 
The completed adjoint $L$-function $\Lambda(s,F,\adjoint)$ is holomorphic on the complex plane. 
\end{theorem}
\begin{proof}
By \cite[Theorem 6.1]{hundley2}, $L(s,F,\adjoint)$ is holomorphic on $\Re(s)\geq 1/2$. If $F$ satisfies the generalized Ramanujan (Selberg) conjecture at the archimedean place, i.e., 
$$\Im(\mu_i(F))=0$$
for $i=1,2,3$, then all the poles of the gamma factors are on $\Re(s)=0$. Therefore, $\Lambda(s,F,\adjoint)$ is holomorphic on $\Re(s)\geq 1/2$. 
If $F$ does not satisfy the generalized Ramanujan conjecture at the archimedean place, 
by \reff{unitary}, 
we have 
$$\{\mu_1(F),\mu_2(F),\mu_3(F)\}=\{\sigma+it,-\sigma+it,-2it\},$$
with $\sigma,t\in \mathbb{R}$ and $0\leq \sigma\leq 5/14$ (Kim-Sarnak bound in \cite[Appendix 2]{kim}). 
If $\Lambda(s,F,\adjoint)$ has a pole in $1/2\le \Re(s)<1$, 
it must be $s=2\sigma$ with $1/4\leq \sigma\leq 5/14$.
As we know from the classical Rankin-Selberg theory,
the completed Rankin-Selberg $L$-function 
$\Lambda(s,F\times \tilde{F})$ is holomorphic in $0<\Re(s)<1$.
Thus from the perspective of 
$$\Lambda(s,F,\adjoint)=\frac{\Lambda(s,F\times \tilde{F})}{\zeta(s)\Gamma_\mathbb{R}(s)},$$
we must have $\zeta(2\sigma)=0$. This is false and a contradiction is reached. 
In conclusion, $\Lambda(s,F,\adjoint)$ is holomorphic on $\Re(s)\geq 1/2$ and by its functional equation it is holomorphic on the whole complex plane.
\end{proof}

We are given by Joseph Hundley the following formula for the adjoint $L$-function on $\GL_3$. This formula appears in a different form in his earlier work \cite[(7)]{hundley}. 
It can be viewed as a special case of Theorem \ref{thm:adjoint-lusztig} for the local part $L_p(s, F, \adjoint)$.  
Hundley derived the formula from \cite{ginzburg}, independently from \cite{hesselink}.
%\cyan{The history part is still undecided.}

\begin{theorem}\label{thm:adjoint-L-function}
For a Hecke-Maass cusp form $F$ with Fourier coefficients $A_F(*,*)$. 
We have for $\Re(s)>1$
\begin{equation}\label{eq:L_p}
L_p(s,F,\adjoint)= \zeta_p(2s)\zeta_p(3s)\sum_{\substack{m_1,m_2=0\\3|m_1-m_2}}^\infty \sum_{j=\max(m_1,m_2)}^{m_1+m_2} \frac{A_F(p^{m_1},p^{m_2})}{p^{js}}
\end{equation}
and $L(s, F, \adjoint)=\prod\limits_{p \text{ is a prime}} L_p(s,F,\adjoint)$.
Moreover, we have for $\Re(s)>1$
\begin{equation}\label{eq:L_sharp}
L(s,F,\adjoint)=\zeta(2s)\zeta(3s)L^\#(s,F,\adjoint)
\end{equation}
and $$L^\#(s,F,\adjoint) 
=\sumad \frac{A_F(n_1,n_2)}{ n^s}.$$
\end{theorem}
\begin{proof}
Obviously \reff{eq:L_p} implies \reff{eq:L_sharp}. 
Lemma \ref{lem:speical_case_KL} and Theorem \ref{thm:adjoint-lusztig} imply 
\reff{eq:L_p}.
\end{proof}

\section{The proof of the main theorem }

In order to prove the main theorem (Theorem \ref{thm:main_thm}), by the Stone-Weierstrass theorem (also the Peter-Weyl theorem as in \cite[Theorem 7.1]{zhou2}), it is sufficient to prove the following one, with effective error term.

\begin{theorem}\label{thm:mainthm_prep}
For integers $l_1,l_2\geq 0$, we have 
$$  \sum\limits_F A_F(p^{l_2},p^{l_1}) h_T(\nu_F)  = \mathfrak{M}^0_{\aleph(l_2,l_1)}(p^{-1}) \bra{\frac{1}{64\pi^5} \int_{\Re(\nu)=0} h_T(\nu) \spec(\nu) d\nu} +\mathcal{O}\bra{T^{\frac{14-4\varepsilon}{3}+\epsilon}p^{\frac{l_1+l_2}{2}+\epsilon}}$$
for $\epsilon>0$.
\end{theorem}

The value of $L(1,F,\adjoint) $ can be approximated by the Dirichlet series in Theorem \ref{thm:adjoint-L-function}.  
\begin{lemma}\label{lemma:approximate_dirichlet}Let $F$ be a Maass cusp form with spectral parameters $(\mu_1(F),\mu_2(F),\mu_3(F))\in \mathfrak{a}^*_\mathbb{C`}$. For $0<\sigma<1$, we have for $X>0$
$$L(1,F,\adjoint)=\zeta(2)\zeta(3) \sumad
\frac{A_F(n_1,n_2)}{n} e\bra{-\frac{n}{X}}+ I_F(X)$$
with 
$$I_F(X)\ll \bra{\prod_{i\neq j}|\mu_i(F)-\mu_j(F)|}^{\frac{1	}{4}+\epsilon} X^{-\frac{1}{2}}.$$
%$$I_F(X)= \mathcal{O}\bra{T^{\frac{3}{2}+\epsilon}X^{-\frac{1}{2}}}.$$
\end{lemma}
\begin{proof}
We begin with
\begin{align*}
 -I_F(X)&=\frac{{\zeta(2)\zeta(3)}}{2\pi i}\int_{(-1/2)}\frac{L(s+1,F,\adjoint)}{\zeta(2s+2)\zeta(3s+3)}\Gamma(s)X^s \dd s\\
&=-L(1,F,\adjoint)+\frac{{\zeta(2)\zeta(3)}}{2\pi i}\int_{(+)}\frac{L(s+1,F,\adjoint)}{\zeta(2s+2)\zeta(3s+3)}\Gamma(s)X^s \dd s\\
&=-L(1,F,\adjoint)+{\zeta(2)\zeta(3)} \sumad \frac{A_F(n_1,n_2)}{n} e\bra{-\frac{n}{X}}
\end{align*}
by the Mellin inversion.
Because $\Lambda(s,F,\adjoint) $ is holomorphic from Theorem \ref{thm:holomorph} and satisfies the functional equation \reff{eq:functional_eq_adjoint}, $L(s,F,\adjoint)$ satisfies  the convexity bound by
 the Phragm\'en-Lindel\"of principle. 
We apply the convexity bound to $L(s+1,F,\adjoint)$ on the vertical line $\Re (s+1)=1/2$ and we get the bound for $I_F(X)$.
\end{proof}

The following theorem improves \cite[Theorem 5]{blomer} and \cite[Theorem 1.3]{goldfeldkontorovich}. It is an application of the Kuznetsov trace formula on $\GL_3$ and it  will be proved in the remaining sections.
\begin{theorem}
\label{thm:ValThm5}
Let $P=m_1 m_2 n_1 n_2\ne 0$, then we have 
\begin{align*}
	\mathcal{C} := \sum_F \frac{A_F(n_1,n_2)\wbar{A_F(m_1, m_2)}}{L(1,F,\adjoint)} h_T\paren{\nu_F} = \;& \Delta+\mathcal{O}\paren{(TP)^\epsilon \paren{T P^{1/2}+T^3 P^{1/6}}},
\end{align*}
where we have 
\begin{align*}
	\Delta = \; \delta_{n=m} \frac{1}{64\pi^5} \int_{\Re(\nu)=0} h_T(\nu) \spec(\nu) d\nu, \qquad \spec(\nu) := \prod_{j=1}^3 \left(3\nu_j \tan\Bigl(\frac{3\pi}{2} \nu_j\Bigr)\right).
\end{align*}
\end{theorem}
\begin{proof}[Proof of Theorem \ref{thm:mainthm_prep}]
We start with the formal sum
\begin{align*}
\sum_F& A_F(p^{l_2},p^{l_1}) h_T(\nu_F)  \\
&= \sum_F \frac{A_F(p^{l_2},p^{l_1})}{L(1,F,\adjoint)} L(1,F,\adjoint) h_T(\nu_F) 
\\
&=\sum_F \frac{A_F(p^{l_2},p^{l_1})}{L(1,F,\adjoint)} \bra{
\zeta(2)\zeta(3)
\sumad \frac{A_F(n_1,n_2)}{n}e\bra{-\frac n X}+I_F(X)} h_T(\nu_F)\\
&=\zeta(2)\zeta(3)\bra{\mathcal{M}+\mathcal{E}_1}+\mathcal{E}_2,
\end{align*}
where the main term is by Theorem \ref{thm:ValThm5}
$$\mathcal M= \bra{\sum_{\max\{l_1,l_2\}\leq  l \leq l_1+l_2} \frac{1}{	p^l} e\bra{-\frac {p^l}{ X}}}\bra{\frac{1}{64\pi^5} \int_{\Re(\nu)=0} h_T(\nu) \spec(\nu) d\nu}$$
and the error terms are ($\mathcal{E}_1$ comes from Theorem \ref{thm:ValThm5})
\begin{align*}
\mathcal{E}_1=&
\sumad \frac{1}{n} e\bra{-\frac{n}{X}}
\mathcal{O}\bra{(TP)^{\epsilon}\bra{TP^{1/2}+T^3P^{1/6}}},\quad \; 
\text{with }P=n_1n_2p^{l_1}p^{l_2}
\\=&\sum_{n=1}^\infty \frac{1}{n} e\bra{-\frac{n}{X}}
\mathcal{O}\bra{(TP)^{\epsilon}\bra{TP^{1/2}+T^3P^{1/6}}}, \quad  \;	
\text{with }P=n^2p^{l_1+l_2}
\\=&\;\mathcal{O}\bra{(TX^2)^\epsilon\bra{TX+T^3X^{1/3}}p^{\frac{l_1+l_2}{2}+\epsilon}}
\end{align*}
and %Lemma \ref{lem:weyllawdifferent}
\begin{align*}
\mathcal{E}_2&=\sum_F \frac{A_F(p^{l_2},p^{l_1})}{L(1,F,\adjoint)}I_F(X) h_T(\nu_F)\\
&\leq \bra{\sum_F \frac{|A_F(p^{l_2},p^{l_1})|^2}{L(1,F,\adjoint)} h_T(\nu_F) }^{1/2}\bra{\sum_F\frac{|I_F(X)|^2}{L(1,F,\adjoint)} h_T(\nu_F) }^{1/2}\\
&\ll \bra{\frac{1}{64\pi^5} \int_{\Re(\nu)=0} h_T(\nu) \spec(\nu) d\nu} p^{\frac{l_1+l_2}{2}+\epsilon}T^{\frac{3}{2}+\epsilon}X^{-\frac 1 2},
\end{align*}
because Lemma \ref{lemma:approximate_dirichlet} and Theorem \ref{thm:ValThm5} imply 
that for $\delta>0$ and  for $X\gg T^{3+\delta}$, we have 
$$\sum_F\frac{|I_F(X)|^2}{L(1,F,\adjoint)} h_T(\nu_F) \ll  
\bra{\frac{1}{64\pi^5} \int_{\Re(\nu)=0} h_T(\nu) \spec(\nu) d\nu}
T^{3 +\epsilon} X^{- 1}.$$
Since we have 
$$\frac{1}{64\pi^5} \int_{\Re(\nu)=0} h_T(\nu) \spec(\nu) d\nu = C T^{5-2\varepsilon}+\mathcal{O}(T^{5-3\varepsilon}) \text{ for some constant }C>0,$$
balancing $\mathcal{E}_1$ and $\mathcal{E}_2$ by taking $X=T^{\frac{11-4\varepsilon}{3}}$, we get 
$\sum_F A_F(p^{l_2},p^{l_1}) h_T(\nu_F)$ is equal to 
$$\zeta(2)\zeta(3)\bra{\sum_{\max\{l_1,l_2\}\leq  l \leq l_1+l_2} \frac{1}{	p^l} }\bra{\frac{1}{64\pi^5} \int_{\Re(\nu)=0} h_T(\nu) \spec(\nu) d\nu} + \mathcal{O}\bra{ T^{\frac {14-4\varepsilon}{3}  +\epsilon} p^{\frac{l_1+l_2}{2}+\epsilon}}.\qedhere
$$
\end{proof}
\nothing{
\begin{lemma}\label{lem:weyllawdifferent}
For $\delta>0$ and  for $X\gg T^{3+\delta}$, we have 
$$\sum_F\frac{|I_F(X)|^2}{L(1,F,\adjoint)} h_T(\nu_F) \ll  
\bra{\frac{1}{64\pi^5} \int_{\Re(\nu)=0} h_T(\nu) \spec(\nu) d\nu}
T^{3 +\epsilon} X^{- 1}.$$
\end{lemma}
\begin{proof}
This is an application of Lemma \ref{lemma:approximate_dirichlet} and Theorem \ref{thm:ValThm5}.
\end{proof}
}

\begin{remark}
The bound for
$$\sum_F\frac{|I_F(X)|^2}{L(1,F,\adjoint)} h_T(\nu_F) \ll  
\bra{\frac{1}{64\pi^5} \int_{\Re(\nu)=0} h_T(\nu) \spec(\nu) d\nu}
T^{3 +\epsilon} X^{- 1}$$
comes from the individual convexity bound for $L(s,F,\adjoint)$ in Lemma \ref{lemma:approximate_dirichlet}. This is far from what can be conjectured. 
Under the generalized Lindel\"of hypothesis, $T^{3+\epsilon}$ can be replaced with $T^\epsilon$ and %$X$ can be taken smaller than $ T^2$. 
greater power saving can be achieved.
%%Then the analysis in this paper is unnecessarily. 
The average of the Lindel\"of hypothesis or the subconvexity bound is not available for this  family of $L$-functions $\{L(s,F,\adjoint):F\text{ is a Hecke-Maass cusp form for } \ssl 3 Z\}$. 
\end{remark}

%%%%%%%%%%%%%%%%%%%%%%%%%%%%%%%%%%%%%%%%%%%%%%%%%%%5
%%%%%%%%%%%%%%%%%%%%%%%%%%%%%%%%%%%%%%%%%%%%%%%%%%%%5
%%%%%%%%%%%%%%%%%%%%%%%%%%%%%%%%%%%%%%%%%%%%%%%%%%%%5
%%%%%%%%%%%%%%%%%%%%%%%%%%%%%%%%%%%%%%%%%%%%%%%%%%%%5
%%%%%%%%%%%%%%%%%%%%%%%%%%%%%%%%%%%%%%%%%%%%%%%%%%%%5
%%%%%%%%%%%%%%%%%%%%%%%%%%%%%%%%%%%%%%%%%%%%%%%%%%%%5
%%%%%%%%%%%%%%%%%%%%%%%%%%%%%%%%%%%%%%%%%%%%%%%%%%%%5
%%%%%%%%%%%%%%%%%%%%%%%%

\section{The Kuznetsov Formula}\label{kuznetsovTF}
In this section we state the Kuznetsov formula for the particular test function $h_T$ described in Section \ref{FWOrtho} above.
We will be brief and refer to \cite[Section 3]{blomerbuttcane1} for more details and notation.  
In particular, we will not require the precise definition of the two Kloosterman sums $\tilde{S}(n_1,n_2,m_1;D_1,D_2)$ and $S(n_1, n_2, m_1, m_2; D_1, D_2)$ given in \cite[Section 5.1 and (1.1)]{blomerbuttcane2}, and we will treat the two Eisenstein series terms, which we denote $\mathcal{E}_\text{max}$ and $\mathcal{E}_\text{min}$, occuring in \cite[Theorem 4]{buttcane} trivially.

For $s = (s_1, s_2) \in \Bbb{C}^2$, $\mu \in \Bbb{C}^3$ with $\mu_1 + \mu_2 + \mu_3 = 0$  define the meromorphic functions
\begin{align}\label{Gsmu}
G(s, \mu) :=& \frac{1}{\Gamma(s_1 + s_2)} \prod_{j=1}^3 \Gamma(s_1 - \mu_j) \Gamma(s_2 + \mu_j), \\
\tilde{G}^{\pm}(s, \mu) :=& \frac{  \pi^{-3s}}{12288 \pi^{7/2}}\Biggl(\prod_{j=1}^3\frac{\Gamma(\frac{1}{2}(s-\mu_j))}{\Gamma(\frac{1}{2}(1-s+\mu_j))} \pm i   \prod_{j=1}^3\frac{\Gamma(\frac{1}{2}(1+s-\mu_j))}{\Gamma(\frac{1}{2}(2-s+\mu_j))} \Biggr),
\end{align}
and the following trigonometric functions
\begin{displaymath}
\begin{split}
& S^{++}(s, \mu) := \frac{1}{24 \pi^2} \prod_{j=1}^3 \cos\left(\frac{3}{2} \pi \nu_j\right),\\
&  S^{+-}(s, \mu) :=  -\frac{1}{32 \pi^2} \frac{\cos(\frac{3}{2} \pi \nu_2)\sin(\pi(s_1 - \mu_1))\sin(\pi(s_2 + \mu_2))\sin(\pi(s_2 + \mu_3))}{\sin(\frac{3}{2} \pi \nu_1)\sin(\frac{3}{2} \pi \nu_3) \sin(\pi(s_1+s_2))}, \\
& S^{-+}(s, \mu) :=-\frac{1}{32 \pi^2}  \frac{\cos(\frac{3}{2} \pi \nu_1)\sin(\pi(s_1 - \mu_1))\sin(\pi(s_1 - \mu_2))\sin(\pi(s_2 + \mu_3))}{\sin(\frac{3}{2} \pi \nu_2)\sin(\frac{3}{2} \pi \nu_3)\sin(\pi(s_1+s_2))}, \\
& S^{--}(s, \mu) := \frac{1}{32 \pi^2}  \frac{\cos(\frac{3}{2} \pi \nu_3) \sin(\pi(s_1 - \mu_2))\sin(\pi(s_2 + \mu_2))}{\sin(\frac{3}{2} \pi \nu_2)\sin(\frac{3}{2} \pi \nu_1)}. 
  \end{split}
\end{displaymath}
Then for $y = (y_1, y_2) \in (\Bbb{R}\setminus \{0\})^2$ with $\text{{\rm sgn}}(y_1) = \alpha_1$, $\text{{\rm sgn}}(y_2) = \alpha_2$, let
\begin{equation}\label{defK}
\begin{split}
 K^{\alpha_1, \alpha_2}_{w_6}(y; \mu)  =    & \int_{-i\infty}^{i\infty} \int_{-i\infty}^{i\infty}  |4\pi^2 y_1|^{-s_1} |4\pi^2 y_2|^{-s_2}  G(s, \mu) S^{\alpha_1, \alpha_2}(s, \mu)\frac{ds_1\, ds_2}{(2\pi i)^2}, 
\end{split}
\end{equation}
and for $y \in \Bbb{R} \setminus \{0\}$ with $\text{{\rm sgn}}(y) = \alpha$, let
\begin{equation}\label{defKw4}
	K_{w_4}(y; \mu) =  \int_{-i\infty}^{i\infty} | y|^{-s} \tilde{G}^{\alpha}(s, \mu) \frac{ds}{2\pi i}.
\end{equation}
The paths of integration must be chosen according to the Barnes convention as in \cite[Definition 1]{blomerbuttcane1}. 
Then for $n_1, n_2, m_1, m_2 \in \Bbb{N}$ and $h_T$ as above we have 
\begin{displaymath}
\begin{split}
&  \mathcal{C} = \Delta + \Sigma_4 + \Sigma_5 + \Sigma_6 - \mathcal{E}_\text{max} - \mathcal{E}_\text{min},
\end{split}
\end{displaymath}
with $\mathcal{C}$ and $\Delta$ as in Theorem \ref{thm:ValThm5},
\begin{displaymath}
\begin{split}
   \Sigma_{4}& = \sum_{\alpha  = \pm 1} \sum_{\substack{D_2 \mid D_1\\  m_2 D_1= n_1 D_2^2}}\frac{ \tilde{S}(-\alpha n_2, m_2, m_1; D_2, D_1)}{D_1D_2} \Phi_{w_4}\left(  \frac{\alpha m_1m_2n_2}{D_1 D_2} \right),  \\
 \Sigma_{5} &= \sum_{\alpha  = \pm 1} \sum_{\substack{     D_1 \mid D_2\\ m_1 D_2 = n_2 D_1^2}} \frac{ \tilde{S}(\alpha n_1, m_1, m_2; D_1, D_2) }{D_1D_2}\Phi_{w_5}\left( \frac{\alpha n_1m_1m_2}{D_1 D_2}\right),\\
\end{split}
\end{displaymath}
\begin{equation}\label{longWeyl}
   \Sigma_6 = \sum_{\alpha_1, \alpha_2 = \pm 1} \sum_{D_1,  D_2  } \frac{S(\alpha_2 n_2, \alpha_1 n_1, m_1, m_2; D_1, D_2)}{D_1D_2} \Phi_{w_6}  \left( - \frac{\alpha_2 m_1n_2D_2}{D_1^2}, - \frac{\alpha_1 m_2n_1D_1}{ D_2^2}\right),
 \end{equation}
and 
\begin{equation}\label{defPhi}
\begin{split}
& \Phi_{w_4}(y) =  \int_{\Re \mu = 0} h_T(\mu) K_{w_4}(y; \mu )\, \text{spec}(\mu) d \mu,\\
& \Phi_{w_5}(y) = \int_{\Re \mu = 0} h_T(\mu) K_{w_4}(-y; -\mu )\, \text{spec}(\mu) d \mu ,\\
& \Phi_{w_6}(y_1, y_2) = \int_{\Re \mu = 0} h_T(\mu) K^{\text{sgn}(y_1), \text{sgn}(y_2)}_{w_6}((y_1, y_2) ;  \mu )\, \text{spec}(\mu) d \mu.
\end{split}
\end{equation}

As in \cite[Section 5.2]{blomerbuttcane2}, we may truncate the sums of Kloosterman sums at some high power of $T$, say $D_1 D_2 \ll T^{100}$, and then replace $h_T$, up to a negligible error, with a real-analytic, Weyl-group invariant function that is compactly supported in $T \Omega'$, where $\Omega' \supseteq \Omega$ is a slightly bigger compact subset not intersecting the Weyl chamber walls, and satisfies
\begin{equation}\label{diffh}
\mathscr{D}_j h_T \ll T^{j(\varepsilon-1)}
\end{equation}
for every differential operator of order $j$.

Theorem \ref{thm:ValThm5} will require strong bounds on the $\Phi_w$ functions, which we provide in the following proposition.
\begin{proposition} \ 
\label{prop:Bounds}

\begin{enumerate}
\item[\emph{(a)}] For $y\in(T^{-100},T^{100})^2$, $\alpha\in\set{\pm 1}^2$,  we have $\Phi_{w_6}(\alpha y) \ll T^{3+\epsilon}$.
\item[\emph{(b)}] For $y\in(T^{-100},T^{100})$, $\alpha\in\set{\pm 1}$, we have $\Phi_{w_4}(\alpha y) \ll T^{3+\epsilon} \paren{y^{1/6}+y^{-1/6}}$.
\end{enumerate}
\end{proposition}

We first note that $K_{w_6}$ has a much simpler sum-of-Mellin-Barnes integral expression than given above.
For $\alpha\in\set{\pm 1}^2$, define
\begin{align}
	G^\alpha(s,\mu) =& \frac{\sqrt{\pi}}{768} \sum_{d\in\set{0,1}^2} \alpha_1^{d_1} \alpha_2^{d_2} (-1)^{d_1 d_2} \frac{\Gamma\paren{\frac{1+d_3-s_1-s_2}{2}}}{\Gamma\paren{\frac{d_3+s_1+s_2}{2}}} \prod_{i=1}^3 \frac{\Gamma\paren{\frac{d_1+s_1-\mu_i}{2}}\Gamma\paren{\frac{d_2+s_2+\mu_i}{2}}}{\Gamma\paren{\frac{1+d_1-s_1+\mu_i}{2}}\Gamma\paren{\frac{1+d_2-s_2-\mu_i}{2}}},
\end{align}
using $d_3 \equiv d_1+d_2\pmod{2}$, $d_3 \in \set{0,1}$.
Then the long-element kernel function has the following expression
\begin{align}
	\tfrac{1}{6} \sum_{w\in W}K^{\alpha_1,\alpha_2}_{w_6}(y;w(\mu)) =& K^{\text{sym}}_{w_6}(\alpha y;\mu),
\end{align}
where
\begin{align}
\label{eq:KsymMB}
	K^{\text{sym}}_{w_6}(\alpha y;\mu) =& \int_{-i\infty}^{i\infty} \int_{-i\infty}^{i\infty} \abs{\pi^2 y_1}^{-s_1} \abs{\pi^2 y_2}^{-s_2} G^\alpha(s,\mu) \frac{ds_1\, ds_2}{(2\pi i)^2},
\end{align}
and the unbounded portion of the $s_1,s_2$ integrals must pass to the left of the zero line.
This identity of functions may be verified by shifting the contours to the left and comparing power series expansions.
Some care must be taken that
\[ \Re(2s_1-s_2)<0, \qquad \Re(2s_2-s_1) < 0, \]
on the unbounded portions to maintain absolute convergence; this requires shifting the contours in stages.

Then using \eqref{defKw4} and the simplified \eqref{eq:KsymMB}, we can prove the following integral representations.
\begin{lemma}
\label{lem:IntRepns}
The integral kernels above may be expressed as
\begin{align}
\label{eq:KsymIntRepn}
	K^{\text{sym}}_{w_6}(\alpha y;\mu) =& \frac{1}{6144\pi} \sum_{\substack{d\in\set{0,1}^2\\ \eta\in\set{\pm1}^3}} (\eta_1\eta_3\alpha_1)^{d_1} (\eta_2\alpha_2)^{d_2}  \int_0^\infty \int_0^\infty \paren{\frac{y_2 z_1^{3/2}}{y_1}}^{-\nu_1} \paren{\frac{y_1 z_2^{3/2}}{y_2}}^{-\nu_2} \\
	& \qquad \sign(1+\eta_1 z_1)^{d_1+d_2} \sign(1+\eta_2 z_2)^{d_1+d_2} \sign(1+\eta_3 z_3)^{d_1+d_2} \nonumber \\
	& \qquad \paren{\frac{2}{\pi}K_0\paren{4\pi \sqrt{\abs{z_4}}}-(-1)^{d_3}Y_0\paren{4\pi \sqrt{\abs{z_4}}}} \frac{dz_1 dz_2}{z_1 z_2}, \nonumber
\end{align}
with
\begin{align}
\label{eq:LongElezs}
	z_3 = \frac{z_2 y_1^2}{z_1 y_2^2}, \qquad z_4 = \paren{1+\eta_1 \sqrt{z_1}} \paren{1+\eta_2 \sqrt{z_2}} \paren{1+\eta_3 \sqrt{z_3}} \frac{y_2}{\sqrt{z_2}}, \qquad d_3=d_1+d_2-2 d_1 d_2,
\end{align}
and
\begin{align}
\label{eq:Kw4IntRepn}
	K_{w_4}(\alpha y;\mu) =& -\frac{1}{2^{13} \, 9\pi^5} \sum_{\substack{d\in\set{0,1}\\ \eta\in\set{\pm1}^2}} (-\eta_1\eta_2\alpha)^d \int_0^\infty \int_0^\infty \sign(1+\eta_2 z_2)^d \abs{1+\eta_1 \sqrt{z_1}}^{-\mu_2} z_1^{-\frac{\mu_1}{2}} z_2^{-\frac{\mu_2}{2}} \\
	& \qquad \abs{z_3}^{\frac{1-2d}{6}} \sin\pi \paren{\tfrac{d}{2}+2\sqrt[3]{\abs{z_3}}} \frac{dz_1 \, dz_2}{z_1 z_2}, \nonumber
\end{align}
with
\begin{align}
\label{eq:w4zs}
	z_3 = \frac{y}{\sqrt{z_1 z_2}} \paren{1+\eta_1 \sqrt{z_1}}^2 \paren{1+\eta_2 \sqrt{z_2}}^3.
\end{align}
The integral \eqref{eq:Kw4IntRepn} converges in the Riemannian sense.
\end{lemma}
Here $K_0$ and $Y_0$ are the usual Bessel functions.

\begin{proof}[Proof of Proposition \ref{prop:Bounds}]
For the compactly supported test function $h_T$ described above, define
\begin{align}
\label{eq:hstarhatdef}
	\widecheck{h}_T\paren{v_1,v_2} = -3 T^{-3} \frac{1}{6144\pi} \int_{-\infty}^\infty \int_{-\infty}^\infty h_T\paren{iu} v_1^{iu_1} v_2^{iu_2} \spec\paren{iu} du_1 \, du_2.
\end{align}
The support of $h_T$ allows us to remove the hyperbolic tangents from $\spec\paren{iu}$ at a negligible cost, and integration by parts in $\widecheck{h}_T\paren{v_1,v_2}$ constrains $v_1$ and $v_2$ to a region $v_i=1+\mathcal{O}\paren{T^{\varepsilon-1+\epsilon}}$.
We have the trivial bound $\widecheck{h}_T\paren{v_1,v_2} \ll T^{2-2\varepsilon}$.

Now for $y\in(T^{-100},T^{100})^2$, $\alpha\in\set{\pm 1}^2$, Lemma \ref{lem:IntRepns} implies
\begin{align}
\label{eq:LongEleIntRepn}
	\Phi_{w_6}(\alpha_1 y_1, \alpha_2 y_2) =& T^3 \sum_{\substack{d\in\set{0,1}^2\\ \eta\in\set{\pm1}^3}} (\eta_1\eta_3\alpha_1)^{d_1} (\eta_2\alpha_2)^{d_2} \int_0^\infty \int_0^\infty \widecheck{h}_T\paren{\frac{y_2 z_1^{3/2}}{y_1},\frac{y_1 z_2^{3/2}}{y_2}} \\
	& \qquad \sign(1+\eta_1 z_1)^{d_1+d_2} \sign(1+\eta_2 z_2)^{d_1+d_2} \sign(1+\eta_3 z_3)^{d_1+d_2} \nonumber \\
	& \qquad \paren{\frac{2}{\pi}K_0\paren{4\pi \sqrt{\abs{z_4}}}-(-1)^{d_3}Y_0\paren{4\pi \sqrt{\abs{z_4}}}} \frac{dz_1 dz_2}{z_1 z_2}. \nonumber
\end{align}
Then using the trivial bound $\ll \log \paren{3+\abs{z_4}^{-1}}$ for the Bessel functions implies $\Phi_{w_6}(\alpha_1 y_1, \alpha_2 y_2) \ll T^{3+\epsilon}$.

For $y\in(T^{-100},T^{100})$, $\alpha\in\set{\pm 1}$, Lemma \ref{lem:IntRepns} again implies
\begin{align}
\label{eq:w4IntRepn}
	\Phi_{w_4}(\alpha y) =& -\frac{T^3}{8 \pi^{11/2}} \sum_{\substack{d\in\set{0,1}\\ \eta\in\set{\pm1}^2}} (-\eta_1\eta_2\alpha)^d \int_0^\infty \int_0^\infty \widecheck{h}_T\paren{\frac{z_1}{\sqrt{z_2}\abs{1+\eta_1 \sqrt{z_1}}},\sqrt{z_1z_2}\abs{1+\eta_1 \sqrt{z_1}}} \\
	& \qquad \sign(1+\eta_2 z_2)^d \abs{z_3}^{\frac{1-2d}{6}} \sin\pi \paren{\tfrac{d}{2}+2\sqrt[3]{\abs{z_3}}} \frac{dz_1 \, dz_2}{z_1 z_2}, \nonumber
\end{align}
and trivially bounding the integral gives $\Phi_{w_4}(\alpha y) \ll T^{3+\epsilon} \paren{y^{1/6}+y^{-1/6}}$.

\end{proof}

The next two sections are devoted to proving the integral representations of Lemma \ref{lem:IntRepns}.

\subsection{The long element weight function}
Starting from \eqref{eq:KsymMB}, substitute $s_1 \mapsto 2s_1+\mu_2$, $s_2 \mapsto 2s_2-\mu_2$ so that
\begin{align*}
	& K^{\text{sym}}_{w_6}(\alpha y;\mu) = \\
	& \frac{\sqrt{\pi}}{768} \sum_{d\in\set{0,1}^2} \alpha_1^{d_1} \alpha_2^{d_2} (-1)^{d_1 d_2} \paren{\frac{y_2}{y_1}}^{\mu_2} \int_{-i\infty}^{i\infty} \int_{-i\infty}^{i\infty} (\pi^2 y_1)^{-2s_1} (\pi^2 y_2)^{-2s_2} \\
	& \qquad \frac{\Gamma\paren{\frac{1+d_3}{2}-s_1-s_2}}{\Gamma\paren{\frac{d_3}{2}+s_1+s_2}} \frac{\Gamma\paren{\frac{d_1}{2}+s_1}\Gamma\paren{\frac{d_2}{2}+s_2}}{\Gamma\paren{\frac{1+d_1}{2}-s_1}\Gamma\paren{\frac{1+d_2}{2}-s_2}} \frac{\Gamma\paren{\frac{d_1}{2}+s_1-\frac{3}{2}\nu_1}\Gamma\paren{\frac{d_2}{2}+s_2+ \frac{3}{2}\nu_1}}{\Gamma\paren{\frac{1+d_1}{2}-s_1+ \frac{3}{2}\nu_1}\Gamma\paren{\frac{1+d_2}{2}-s_2- \frac{3}{2}\nu_1}} \\
	& \qquad \frac{\Gamma\paren{\frac{d_1}{2}+s_1+ \frac{3}{2}\nu_2}\Gamma\paren{\frac{d_2}{2}+s_2- \frac{3}{2}\nu_2}}{\Gamma\paren{\frac{1+d_1}{2}-s_1- \frac{3}{2}\nu_2}\Gamma\paren{\frac{1+d_2}{2}-s_2+ \frac{3}{2}\nu_2}} \frac{ds_1\, ds_2}{(2\pi i)^2}.
\end{align*}

\begin{lemma}
For $\Re(s_1-u),\Re(s_2+u)>0$,$\Re(s_1+s_2)<\frac{1}{2}$, we have
\begin{align}
\label{eq:MeijerGInv}
	& \sqrt{\pi}\frac{\Gamma\paren{\frac{1+d_3}{2}-s_1-s_2}\Gamma\paren{\frac{d_1}{2}+s_1- u}\Gamma\paren{\frac{d_2}{2}+s_2+ u}}{\Gamma\paren{\frac{d_3}{2}+s_1+s_2}\Gamma\paren{\frac{1+d_1}{2}-s_1+ u}\Gamma\paren{\frac{1+d_2}{2}-s_2- u}} =\\
	& \qquad \frac{(-1)^{d_1 d_2}}{2} \sum_{\eta=\pm1} \eta^{d_1} \int_0^\infty \sign(1+\eta z)^{d_1+d_2} \abs{1+\eta \sqrt{z}}^{-2s_1-2s_2} z^{-u+s_1} \frac{dz}{z}, \nonumber \\
\label{eq:MeijerG}
	& \frac{1}{2\pi i}\int_{-i\infty}^{i\infty} \frac{\Gamma\paren{\frac{d_1}{2}+s_1- u}\Gamma\paren{\frac{d_2}{2}+s_2+ u}}{\Gamma\paren{\frac{1+d_1}{2}-s_1+ u}\Gamma\paren{\frac{1+d_2}{2}-s_2- u}} z^{u} du = \\
	& \qquad \frac{(-1)^{d_1 d_2}}{2\sqrt{\pi}} z^{s_1} \frac{\Gamma\paren{\frac{d_3}{2}+s_1+s_2}}{\Gamma\paren{\frac{1+d_3}{2}-s_1-s_2}} \sum_{\eta=\pm1} \eta^{d_1} \sign(1+\eta z)^{d_1+d_2}\abs{1+\eta \sqrt{z}}^{-2s_1-2s_2}, \nonumber
\end{align}
\end{lemma}
\begin{proof}
By elementary substitutions, the right-hand side of \eqref{eq:MeijerGInv} is given by the sum of beta functions
\begin{align*}
	(-1)^{d_1} B(1-2s_1-2s_2,2s_1-2u)+(-1)^{d_2} B(1-2s_1-2s_2,2s_2+2u)+B(2s_1-2u,2s_2+2u).
\end{align*}
Then \eqref{eq:MeijerGInv} follows by applying reflection to the gamma functions in the denominators and trigonometry, and \eqref{eq:MeijerG} follows from Mellin inversion.
Note that Mellin inversion produces an integral which converges in the Riemannian sense; the contour may then be deformed so the unbounded portion passes to the left of the zero line for absolute convergence.
\end{proof}

Plugging \eqref{eq:MeijerGInv} into the preceeding equation, and substituting $s_1 \mapsto s_1-s_2$,
\begin{align*}
	& K^{\text{sym}}_{w_6}(\alpha y;\mu) = \\
	& \frac{1}{3072\sqrt{\pi}} \sum_{\substack{d\in\set{0,1}^2\\\eta\in\set{\pm1}^2}} (\eta_1\alpha_1)^{d_1} (\eta_2\alpha_2)^{d_2}  (-1)^{d_1 d_2} \paren{\frac{y_2}{y_1}}^{\mu_2} \int_{-i\infty}^{i\infty} \int_0^\infty \int_0^\infty \sign(1+\eta_1 z_1)^{d_1+d_2} \sign(1+\eta_2 z_2)^{d_1+d_2} \\
	& \qquad (\pi^2 y_1)^{-2s_1} \frac{\Gamma\paren{\frac{d_3}{2}+s_1}}{\Gamma\paren{\frac{1+d_3}{2}-s_1}} \int_{-i\infty}^{i\infty} \paren{\frac{z_2 y_1^2}{z_1 y_2^2}}^{s_2} \frac{\Gamma\paren{\frac{d_1}{2}+s_1-s_2}\Gamma\paren{\frac{d_2}{2}+s_2}}{\Gamma\paren{\frac{1+d_1}{2}-s_1+s_2}\Gamma\paren{\frac{1+d_2}{2}-s_2}}  \frac{ds_2}{2\pi i}\\
	& \qquad \abs{1+\eta_1 \sqrt{z_1}}^{-2s_1} \abs{1+\eta_2 \sqrt{z_2}}^{-2s_1} z_1^{-\frac{3}{2}\nu_1+s_1} z_2^{-\frac{3}{2}\nu_2} \frac{dz_1 dz_2}{z_1 z_2} \frac{ds_1}{2\pi i}.
\end{align*}

Using $z_3$ as in \eqref{eq:LongElezs} and taking $s_2\to0$ in \eqref{eq:MeijerG}, we have
\begin{align*}
	K^{\text{sym}}_{w_6}(\alpha y;\mu) =& \frac{1}{6144\pi} \sum_{\substack{d\in\set{0,1}^2\\ \eta\in\set{\pm1}^3}} (\eta_1\eta_3\alpha_1)^{d_1} (\eta_2\alpha_2)^{d_2}  \int_0^\infty \int_0^\infty \paren{\frac{y_2 z_1^{3/2}}{y_1}}^{-\nu_1} \paren{\frac{y_1 z_2^{3/2}}{y_2}}^{-\nu_2} \\
	& \qquad \sign(1+\eta_1 z_1)^{d_1+d_2} \sign(1+\eta_2 z_2)^{d_1+d_2} \sign(1+\eta_3 z_3)^{d_1+d_2} \\
	& \qquad \int_{-i\infty}^{i\infty} \frac{\Gamma\paren{\frac{d_3}{2}+s_1}^2}{\Gamma\paren{\frac{1+d_3}{2}-s_1}^2} \abs{1+\eta_3 \sqrt{z_3}}^{-2s_1} \\
	& \qquad \abs{1+\eta_1 \sqrt{z_1}}^{-2s_1} \abs{1+\eta_2 \sqrt{z_2}}^{-2s_1} \paren{\frac{\pi^4 y_2^2}{z_2}}^{-s_1} \frac{ds_1}{2\pi i} \frac{dz_1 dz_2}{z_1 z_2}.
\end{align*}

Then \eqref{eq:KsymIntRepn} follows by applying the following lemma.
\begin{lemma}
For $z > 0$, we have
\begin{align*}
	\frac{1}{2\pi i}\int_{-i\infty}^{i\infty} \frac{\Gamma\paren{\frac{d_3}{2}+s_1}^2}{\Gamma\paren{\frac{1+d_3}{2}-s_1}^2} z^{-s_1} ds_1 =& \frac{2}{\pi}K_0(4z^{1/4})-(-1)^{d_3}Y_0(4z^{1/4}).
\end{align*}
\end{lemma}
\begin{proof}
The $d_3=0$ case is given in \cite[6.422.16]{GradshteynRyzhik} or \cite[8.4.23.15]{Prudnikov}, but the $d_3=1$ case seems to be missing from the literature.
So from \cite[6.561.15 and 6.561.16]{GradshteynRyzhik}, we have
\begin{align*}
	2\int_0^\infty \paren{2 K_0(2x)-(-1)^d \pi Y_0(2x)} x^{u-1} dx =& \Gamma\paren{\frac{u}{2}}^2\paren{1+(-1)^d\cos\frac{\pi u}{2}}.
\end{align*}
By the half-angle and duplication formulae, the right-hand side is
\begin{align*}
	\frac{\pi}{2} 2^u \frac{\Gamma\paren{\frac{d+u/2}{2}}^2}{\Gamma\paren{\frac{1+d-u/2}{2}}^2},
\end{align*}
and the claim follows by Mellin inversion.
\end{proof}

\subsection{The $w_4$ weight function}
In \eqref{defKw4}, we write $\mu_3=-\mu_1-\mu_2$ and apply \eqref{eq:MeijerGInv} with $u=\frac{\mu_1}{2}$, giving
\begin{align*}
	K_{w_4}(\alpha y;\mu) =& \frac{1}{2^{13} \, 3\pi^4} \sum_{\substack{d\in\set{0,1}\\ \eta_1\in\set{\pm1}}} (-\eta_1\alpha)^d \int_{-i\infty}^{i\infty} \abs{\pi^3 y}^{-s} \frac{\Gamma\paren{\frac{2s+\mu_2}{2}}\Gamma\paren{\frac{d+s-\mu_2}{2}}}{\Gamma\paren{\frac{1-2s-\mu_2}{2}}\Gamma\paren{\frac{1+d-s+\mu_2}{2}}} \\
	& \qquad \int_0^\infty \abs{1+\eta_1 \sqrt{z_1}}^{-2s-\mu_2} z_1^{\frac{-\mu_1+s}{2}} \frac{dz_1}{z_1} \frac{ds}{2\pi i},
\end{align*}
and again with $u=\frac{\mu_2}{2}$ giving
\begin{align*}
	K_{w_4}(\alpha y;\mu) =& \frac{1}{2^{14} \, 3\pi^{9/2}} \sum_{\substack{d\in\set{0,1}\\ \eta\in\set{\pm1}^2}} (-\eta_1\eta_2\alpha)^d \int_0^\infty \int_0^\infty \sign(1+\eta_2 z_2)^d \abs{1+\eta_1 \sqrt{z_1}}^{-\mu_2} z_1^{-\frac{\mu_1}{2}} z_2^{-\frac{\mu_2}{2}} \\
	& \qquad \int_{-i\infty}^{i\infty} \frac{\Gamma\paren{\frac{d+3s}{2}}}{\Gamma\paren{\frac{1+d-3s}{2}}} \abs{1+\eta_1 \sqrt{z_1}}^{-2s} \abs{1+\eta_2 \sqrt{z_2}}^{-3s} \abs{\pi^3 y}^{-s} z_1^{\frac{s}{2}} z_2^{\frac{s}{2}} \frac{ds}{2\pi i} \frac{dz_1 \, dz_2}{z_1 z_2}.
\end{align*}

We have \cite[6.422.9]{GradshteynRyzhik},
\[ \int_{-i\infty}^{i\infty} \frac{\Gamma\paren{-u+s}}{\Gamma\paren{1-s}} x^{-s} \frac{ds}{2\pi i} = J_{u}(2\sqrt{x}), \]
which holds for $\Re(u) > 0$, but extends to all $u$ by analytic continuation (deforming the contour for convergence).
Using the known values of $J_{\pm1/2}$ (see \cite[8.464]{GradshteynRyzhik}),
\[ \int_{-i\infty}^{i\infty} \frac{\Gamma\paren{d-\frac{1}{2}+s}}{\Gamma\paren{1-s}} x^{-s} \frac{ds}{2\pi i} = \pi^{-1/2} x^{-1/4} \sin\paren{\tfrac{\pi d}{2}+2\sqrt{x}}, \]
and \eqref{eq:Kw4IntRepn} follows.

\section{Proof of Theorem \ref{thm:ValThm5}}
By \cite[Lemma 9]{blomerbuttcane1} and the argument of \cite[Section 5.2]{blomerbuttcane2}, the long element term sums over $D_1 D_2 \ll P T^{-4} (TP)^\epsilon$, and we apply the proof of \cite[Proposition 3]{blomer} and the bound $T^{3+\epsilon}$ for $\Phi_{w_6}$ from Proposition \ref{prop:Bounds}a.
Similarly, by \cite[Lemma 8]{blomerbuttcane1}, the $w_4$ term sums over $D_1 D_2 \ll P T^{-3} (TP)^\epsilon$, and we use the bound $T^{3+\epsilon} \paren{y^{1/6}+y^{-1/6}}$ for $\Phi_{w_4}$ from Proposition \ref{prop:Bounds}b, in combination with $y \gg T^{3-\epsilon}$, again from \cite[Lemma 8]{blomerbuttcane1}, and Larsen's bound on the $w_4$ Kloosterman sum \cite[Appendix]{BFG}; the treatment of the $w_5$ term is identical.
As in the appendix to \cite{blomer}, the contribution of the Eisenstein series terms is $\mathcal{O}\paren{T^{3+\epsilon} P^{\theta+\epsilon}}$ with $\theta=\frac{7}{64} < \frac{1}{6}$.

\section*{Acknowledgment}
The authors would like to thank Joseph Hundley, without whom this work would not be possible. The authors would like to thank Valentin Blomer for reading the manuscript.
The authors would like to thank Wenzhi Luo and Yiannis Sakellaridis.
%$\mathscr{ABCDEFGHIJKLMNOPQRSTUVWXYZ}$
%\textsc{AaBbCcDdEeFf}
\bibliographystyle{alpha}
\bibliography{bibib_fz,bibib_jb}

\noindent
{\scshape{Jack Buttcane}}\\
Department of Mathematics\\
State University of New York at Buffalo\\
Buffalo, NY\\
\\
{\scshape{Fan Zhou}} \\
Mathematics \& Statistics\\
University of Maine\\
Orono, ME\\

\end{document}